\pdfoutput=1 

\documentclass[reqno,11pt]{amsart}

\usepackage[margin=1in,letterpaper,portrait]{geometry}

\usepackage{microtype} 
\usepackage{
amssymb,
amsmath,
amsthm,
amsbsy,
graphicx,
mathtools,
float, 
tikz, 
}
\usepackage{ytableau}
\usepackage[vcentermath,enableskew]{youngtab}
\usepackage{enumerate}
\usepackage[utf8]{inputenc}
\usepackage[colorlinks=true, pdfstartview=FitV, 
linkcolor=violet, 
citecolor=blue,
filecolor=violet, urlcolor=violet]{hyperref}

\theoremstyle{plain} 
\newtheorem{theorem}{Theorem}[section]
\newtheorem{lemma}[theorem]{Lemma}
\newtheorem{proposition}[theorem]{Proposition}
\newtheorem{corollary}[theorem]{Corollary}

\theoremstyle{definition} 
\newtheorem{definition}[theorem]{Definition}
\newtheorem{example}[theorem]{Example}

\theoremstyle{remark} 
\newtheorem{remark}[theorem]{Remark}

\numberwithin{equation}{section}
\counterwithout{figure}{section}

\newcommand{\arxiv}[1]{\href{http://arxiv.org/abs/#1}{\texttt{arXiv:#1}}}

\renewcommand{\P}{\text{P}}
\DeclareMathOperator{\Q}{Q}
\newcommand{\rw}[1]{\left[{#1}\right]} 
\newcommand{\supp}{\textnormal{\textsf{supp}}} 
\newcommand{\Rowone}{{\sf{{Row}}}_1}
\newcommand{\Rowtwo}[1]
{{\sf{{Row}}}_2(\P(#1))}
\newcommand{\RowTwo}
{{\sf{Row}}_2}
\newcommand{\RowtwoQ}[1]
{{\sf{{Row}}}_2(\Q(#1))}
\newcommand{\balanced}{\textnormal{\textsf{uncrowded}}} 
\newcommand{\unbalanced}{\textnormal{\textsf{crowded}}} 
\newcommand{\crowded}{\textnormal{\textsf{crowded}}} 
\newcommand{\core}[1]{\hat{#1}}

\title{RSK tableaux and the weak order on fully commutative permutations}
\date{}

\author[Gunawan]{Emily Gunawan$^*$} 
\address{Department of Mathematics, The University of Oklahoma, Norman, OK, USA}
\email{egunawan@ou.edu}
\thanks{$^*$ Research partially supported by the Isaac Newton
Institute for Mathematical Sciences (funded by EPSRC Grant Number EP/R014604/1) during the programme Cluster algebras and representation theory.}

\author[Pan]{Jianping Pan} 
\address{Department of Mathematics, North Carolina State University, Raleigh, NC, U.S.A.}
\email{jpan9@ncsu.edu}

\author[Russell]{Heather M. Russell} 
\address{Department of Mathematics and Statistics, University of Richmond, Richmond, VA, USA}
\email{hrussell@richmond.edu}

\author[Tenner]{Bridget Eileen Tenner$^\dagger$}
\address{Department of Mathematical Sciences, DePaul University, Chicago, IL, USA}
\email{bridget@math.depaul.edu}
\thanks{$^\dagger$Research partially supported by NSF Grant DMS-2054436.}


\begin{document}
\begin{abstract}
For each fully commutative permutation, we construct a ``boolean core,'' which is the maximal boolean permutation in its principal order ideal under the right weak order. We partition the set of fully commutative permutations into the recently defined crowded and uncrowded elements, distinguished by whether or not their RSK insertion tableaux satisfy a sparsity condition. We show that a fully commutative element is uncrowded exactly when it shares the RSK insertion tableau with its boolean core.
We present the dynamics of the right weak order on fully commutative permutations, with particular interest in when they change from uncrowded to crowded. In particular, we use consecutive permutation patterns and descents to characterize the minimal crowded elements under the right weak order. 

\end{abstract}
\maketitle

\section{Introduction}
First introduced in~\cite{stembridge96}, the fully commutative elements of a Coxeter group have the property that every pair of reduced words are related by a sequence of commutation relations. This set of objects is combinatorially rich and has been studied extensively (see, for example,~\cite{mpps, nadeau, stembridge98}). A permutation is fully commutative if and only if it avoids the pattern 321 \cite{billey-jockusch-stanley}, and the fully commutative permutations are exactly those with fewer than three rows in their Robinson--Schensted--Knuth (RSK) tableaux \cite{Sch61}. In this paper, following up on recent work in \cite{GPRT1}, we examine the interplay between reduced words and RSK tableaux for fully commutative permutations 
and analyze the set of fully commutative permutations under the weak order.

Our previous work, which is a companion to this paper, 
proves that the RSK insertion tableaux for boolean permutations satisfy a certain sparsity condition that we call {\it uncrowded} \cite{GPRT1}. Boolean permutations are an important subset of fully commutative permutations, characterized by the fact that their principal order ideals in the Bruhat order are isomorphic to boolean algebras. Motivated by those results, we call a fully commutative permutation with an uncrowded insertion tableau an \emph{uncrowded permutation}. In other words, an uncrowded fully commutative permutation shares its insertion tableau with some boolean element.  A fully commutative permutation that is not uncrowded is called \emph{crowded}. Central to this paper is the partition of the set of fully commutative permutations into crowded and uncrowded elements.

For each fully commutative element $w$, we identify a particular boolean element $\core{w}$ that is below $w$ in the weak order and has the same support as $w$; we call this $\hat{w}$ the \emph{boolean core} of $w$ (Theorem~\ref{thm:fully commutative decomposes via boolean}). 
We then view the fully commutative permutation $w$ as an ``elongation" of its boolean core, and we investigate the evolution of RSK insertion tableaux along chains of fully commutative elements in the right weak order. We prove that the second rows of insertion tableaux obey a containment property along covering relations in the right weak order (Theorem~\ref{thm:growing fully commutative means row 2 subsets}).

Applying this containment property, we show that if two fully commutative elements with the same support satisfy a covering relation in the right weak order and have different insertion tableaux then the larger one is necessarily crowded (Theorem~\ref{thm:adding length and box but not support to uncrowded makes crowded}). This has two important implications. First, a fully commutative element is uncrowded exactly when it has the same insertion tableau as its boolean core (Corollary~\ref{cor.uncrowded}).  
Second, within the set of fully commutative permutations under the right weak order, the uncrowded permutations form an order ideal and the crowded permutations form a dual order ideal (Lemma~\ref{lem:balanced order ideal NEW}). Thus, knowing the minimal crowded elements in the poset is, in fact, enough information to identify each fully commutative element as being either crowded or uncrowded. Our final result, Theorem~\ref{thm:minimal crowded iff NEW}, proves a set of necessary and sufficient conditions for a fully commutative permutation to be minimal in the dual order ideal of crowded permutations.

This paper is organized as follows. Section \ref{sec:background} provides necessary background information and notation including several results from our companion paper on boolean RSK tableaux. 
Section \ref{sec:fully comm and the weak order} defines the boolean core of a fully commutative element and proves a containment property for RSK tableaux under the right weak order. 
Section \ref{sec:adding to balanced} explores covering relations between fully commutative elements in the right weak order when the two permutations have the same support but different insertion tableaux. 
Finally, Section \ref{sec:minimal crowded permutations} characterizes the minimal elements of the dual order ideal of crowded fully commutative permutations in the right weak order, thus providing the key to classifying each fully commutative permutation as being either crowded or uncrowded.

\section{Background and notation}\label{sec:background}

Denote the symmetric group on $n$ elements by $S_n$. 
For a permutation $w\in S_n$, we use the \emph{one-line notation} $w = w(1)w(2)\cdots w(n)$ to represent $w$. For each $i \in \{1, \ldots, n-1\}$, we write $s_i \in S_n$ to denote the \emph{simple reflection} (or \emph{adjacent transposition}) 
that swaps $i$ and $i+1$ and fixes all other letters. Every permutation can be expressed as a product of simple reflections. Given $w\in S_n$, the minimum number of simple reflections among all such expressions for $w$ is called the \emph{(Coxeter) length} of $w$, and is denoted by $\ell(w)$. An \emph{inversion} in the one-line notation for $w$ is a pair of positions $i<j$ such that $w(i)>w(j)$. 
It is often convenient to recognize that $\ell(w)$ is the number of inversions in the one-line notation for $w$.
A \emph{reduced decomposition} of $w$ is an expression $w=s_{i_1} \cdots s_{i_{\ell(w)}}$ realizing the Coxeter length of $w$. To simplify notation, we refer to such a  decomposition via its \emph{reduced word} $\rw{i_1 \cdots i_{\ell(w)}}$. Let $R(w)$ denote the set of reduced words for $w$.

The set of letters appearing in reduced words of a permutation $w$ is the \emph{support} $\supp(w)$ of $w$. For example, consider  $w = 51342 = s_{4}s_{2}s_{3}s_{2}s_{4}s_{1} \in S_5.$ Because $w$ has six inversions, we see that $\ell(w)=6$ and $\rw{423241}\in R(w)$. 

The following technical lemma is related to the support of a permutation. It introduces a pair of values $M$ and $m$ which depend on the choice of $v\in S_n$ and $i\in \{1,\dots,n-1 \}$. These values play a central role in the arguments in Section~\ref{sec:adding to balanced}. 
\begin{lemma}\cite[Lemma~2.8]{Ten12}\label{lem:Ten12}
Fix a permutation $v\in S_n$ and $i\in \{1,\dots,n-1 \}$, and let $M := \max\{v(j) : j \le i\}$ 
 and $m := \min\{v(j) : j \ge i+1\}$.
Then the following statements are equivalent:
\begin{itemize}
    \item $i\in \supp(v)$,
    \item $\{v(1),\dots,v(i)\}\neq \{1,2,\dots,i\}$,
    \item $\{v(i+1),\dots,v(n)\}\neq \{i+1,i+2,\dots,n\}$,
    \item $M>i$,
    \item $m<i+1$,
    \item $M>m$.
\end{itemize}
\end{lemma}

The \emph{right weak order}, denoted by $\leq$, is a partial order
on $S_n$ obtained by taking the transitive closure of the cover relation $w < w s_i$ whenever $\ell(w) < \ell(w s_i)$. We use $w<w'$ to denote when $w\le w'$ and $w\neq w'$.
The \emph{left weak order} is defined analogously, with left multiplication by $s_i$ instead of right. 
In each order, the minimum element is the identity permutation and the maximum element is the \emph{long element} $n(n-1)\cdots 21$.
More details on the weak order can be found in, for example,~\cite[Section~3.1]{BB05}.

An \emph{order ideal} of a poset 
is a subset $C$ such that if $y \in C$ and $x \leq y$, then $x \in C$. A \emph{dual order ideal} (or order filter, or upper order ideal) of a poset is a subset $C$ such that if $x \in C$ and $x \leq y$, then $y \in C$. 

\subsection{Fully commutative permutations and boolean permutations}

Let $m\leq n$. The permutation $w\in S_n$ is said to \emph{contain the pattern} $\sigma\in S_m$ if $w$ has a (not necessarily contiguous) subsequence
whose elements are 
in the same relative order as $\sigma$.  In the case that $w$ does not contain $\sigma$, we say $w$ \emph{avoids} $\sigma$. For instance, the permutation 
$w$ = 314592687 contains the pattern 1423 because the subsequence 1927 (among others) has the same relative order as 1423. On the other hand, $w$ avoids 3241 since it has no subsequences that follow the pattern 3241. We note also that the inversions of a permutation are exactly the instances of 21-patterns.

For $|i-j| > 1$, simple reflections satisfy \emph{commutation relations} of the form $s_i s_j=s_j s_i$. An application of a commutation relation to a product of simple reflections is called a \emph{commutation move}. 
In the context of reduced words, 
we will say adjacent letters $i$ and $j$ in a reduced word \emph{commute} when $|i - j| > 1$. For a reduced word $\rw{u}$ of a permutation, the equivalence class of all words obtained from $\rw{u}$ by sequences of commutation moves is called the \emph{commutation class} of $\rw{u}$.
A permutation is called \emph{fully commutative} if all of its reduced words form a single commutation class. As the following proposition shows, fully commutative permutations can be characterized in terms of pattern avoidance.

\begin{proposition}[\cite{billey-jockusch-stanley}]
\label{prop:tfae fully commutative}
Let $w$ be a permutation.
The following are equivalent:
\begin{itemize}
\item $w$ is fully commutative,
\item $w$ avoids the pattern 321,
\item no reduced word of $w$ contains $i(i+1)i$ as a factor, for any $i$, and 
\item no reduced word of $w$ contains $(i+1)i(i+1)$ as a factor, for any $i$.
\end{itemize}
\end{proposition}

Boolean permutations are an important subset of the set of fully commutative permutations. 
The following result gives a description of boolean permutations analogous to that of Proposition~\ref{prop:tfae fully commutative}.

\begin{proposition}[\cite{tenner-patt-bru}]\label{prop:tfae boolean}
Let $w$ be a permutation.
The following are equivalent:
\begin{itemize}
\item $w$ is boolean,
\item $w$ avoids the pattern 321 and 3412,
\item there is a reduced word of $w$ that consists of all distinct letters, and
\item every reduced word of $w$ consists of all distinct letters.
\end{itemize}
\end{proposition}

\subsection{Heaps and commutation class}
\label{section:heaps and commutation class}
\label{sec:heap of boolean permutation}

In this section, we review the classical theory of heaps, which was used in~\cite{Ste96} to study fully commutative elements of a Coxeter group. 
For a detailed list of attributions on the theory of heaps, see~\cite[Solutions to Exercise 3.123(ab)]{Sta12}.

Given a reduced word $\rw{u}$ of a permutation, we can associate to $\rw{u}$ a \emph{heap}, a poset whose elements are labeled by the simple reflections in $\rw{u}$. 
A \emph{heap diagram} is the Hasse diagram for a heap in which poset elements are replaced by their labels.

\begin{definition}
\label{defn:heap of a reduced word}
Given an arbitrary reduced word $\rw{u}=\rw{u_1 \cdots u_\ell}$ of a permutation, 
consider the
partial order $\preccurlyeq$ on the set $\{ 1, \dots, \ell\}$ obtained via the transitive closure of the relations 
\[x
\prec
y\]
for $x < y$ such that
$|u_x-u_y| \leq 1$. 
For each $1 \leq x \leq \ell$, the \emph{label} of the poset element $x$ is  $u_x$. This labeled poset is called the \emph{heap} for $\rw{u}$. The Hasse diagram for this poset with elements $\{1,\ldots, \ell\}$ replaced by their labels is called the \emph{heap diagram} for $\rw{u}$.
\end{definition}

The following lemma follows directly from this definition.

\begin{lemma}\label{lem:label cover}
Let $\rw{u}$ be an arbitrary reduced word for a permutation, and let $x<y$ be elements of the heap for $\rw{u}$.  If $y$ covers $x$, then the labels of $x$ and $y$ differ by exactly one.
\end{lemma}

Note that a heap is, in some sense, a partial ordering on the multiset of simple reflections occurring in a reduced word. For a fully commutative permutation, the heap structure on this multiset is, in fact, independent of the choice of reduced word (see Proposition~\ref{prop:set of labeled linear extensions is commutativity class}). Throughout this paper, for a fully commutative permutation $w$, we will use $H_w$ to denote both the heap diagram for $w$ and the poset of simple reflections of any reduced word $\rw{u}$ of $w$. The context should make it clear to which object $H_w$ refers.

From a linear extension of the heap, one can define a \emph{labeled linear extension} essentially by replacing elements of the heap with their labels.

\begin{definition}\label{def:Labeled linear extensions}
A \emph{labeled linear extension} of the heap of a reduced word $\rw{u}=\rw{u_1 \cdots u_\ell}$
is a word $\rw{u_{\pi(1)} \cdots u_{\pi(\ell)}}$, 
where 
 $\pi = \pi(1) \cdots \pi(\ell)$ 
 is a total order
 on $\{1,\ldots, \ell\}$ that is consistent with the structure of the heap. That is, $\pi(x) \prec \pi(y)$ implies $x<y$. 
\end{definition}

As the next proposition illustrates, labeled linear extensions are related to reduced words and commutation classes.

\begin{proposition}
[{\cite[Proof of Proposition~2.2]{Ste96} and~\cite[Solutions to Exercise 3.123(ab)]{Sta12}}]
\label{prop:set of labeled linear extensions is commutativity class}
Given a reduced word $\rw{u}$,  the set of labeled linear extensions of the heap for $\rw{u}$ is the commutation class of $\rw{u}$.
\end{proposition}

By definition, a fully commutative permutation has exactly one commutation class. Hence Proposition~\ref{prop:set of labeled linear extensions is commutativity class} implies that given any reduced word $\rw{u}$ for a fully commutative permutation $w$, the set of labeled linear extensions of the heap for $\rw{u}$ is exactly $R(w)$, the set of reduced words of $w$.

\begin{example}
\label{ex:heap345619278}
The heap diagram $H_w$ of the fully commutative permutation $w=345619278 \in S_9$
is depicted in Figure~\ref{fig:heap345619278}. 
Two of the labeled linear extensions correspond to the reduced words $\rw{87234561234}$ and $\rw{23451234876}$.
\end{example}

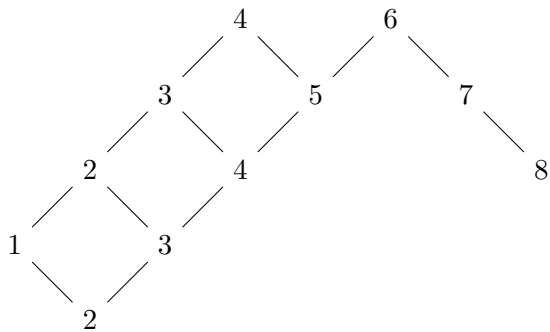
\begin{figure}[htb]
\begin{tikzpicture}[xscale=1,yscale=1,>=latex]
\def\posetedgecolor{blue}
\def\posetedgenegativeslope{red}
\node(1) at (1,1) {$1$}; 
\node(2) at (2,0) {$2$}; 
\node(2-1) at (2,2) {$2$}; 
\node(3) at (3,1) {$3$}; 
\node(3-1) at (3,3) {$3$}; 
\node(4) at (4,2) {$4$};  
\node(4-1) at (4,4) {$4$};  
\node(5) at (5,3) {$5$};  
\node(6) at (6,4) {$6$}; 
\node(7) at (7,3) {$7$}; 
\node(8) at (8,2) {$8$}; 

\draw[-] (1) -- (2); 
\draw[-] (2) -- (3); 
\draw[-] (3) -- (4); 
\draw[-] (4) -- (5); 
\draw[-] (5) -- (6); 
\draw[-] (6) -- (7);
\draw[-] (7) -- (8);

\draw[-] (1) -- (2-1) -- (3-1) -- (4-1);

\draw[-] (3) -- (2-1);
\draw[-] (4) -- (3-1);
\draw[-] (5) -- (4-1);
\end{tikzpicture}
\caption{The heap diagram for the fully commutative permutation $345619278 \in S_9$.}
\label{fig:heap345619278}
\end{figure}

Proposition~\ref{prop:tfae boolean} states that a boolean permutation is a fully commutative permutation with no repeated letters in any of its reduced words. In the sense of heaps, this means that there are no two elements corresponding to the same simple reflection. 
For boolean-specific descriptions of heaps, see \cite[Section~2.2]{GPRT1}.

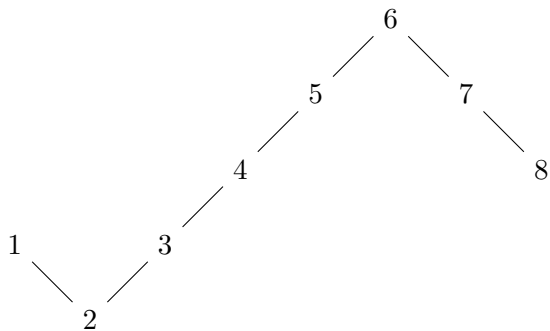
\begin{figure}[htb]
\begin{tikzpicture}[xscale=1,yscale=1,>=latex]
\node(1) at (1,1) {$1$}; 
\node(2) at (2,0) {$2$}; 
\node(3) at (3,1) {$3$}; 
\node(4) at (4,2) {$4$};  
\node(5) at (5,3) {$5$};  
\node(6) at (6,4) {$6$}; 
\node(7) at (7,3) {$7$}; 
\node(8) at (8,2) {$8$}; 

\draw[-] (1) -- (2); 
\draw[-] (2) -- (3); 
\draw[-] (3) -- (4); 
\draw[-] (4) -- (5); 
\draw[-] (5) -- (6); 
\draw[-] (6) -- (7);
\draw[-] (7) -- (8);
\end{tikzpicture}
\caption{The heap diagram for the boolean permutation 
$314569278 \in S_9$. 
}\label{fig:fence_reduced_word21873456}
\end{figure}

 \subsection{Robinson--Schensted--Knuth tableaux} 
The well-known Robinson--Schensted--Knuth (RSK) insertion algorithm, 
as described in~\cite{Sch61}, 
is 
a bijection 
\[{w \mapsto (\P(w),\Q(w))}\]
from $S_n$ onto pairs of standard tableaux of size $n$ having identical shape. 
The tableau $\P(w)$ is  called the \emph{insertion tableau} of $w$, and the tableau $\Q(w)$ is the \emph{recording tableau} of $w$. 
 The shape of these tableaux is the \emph{RSK partition of $w$}.
We will also write $\P_i(w)$ to denote the partial insertion tableau constructed by the first $i$ letters in the one-line notation for $w$. For more details, see for example~\cite[Section~7.11]{sta99}. 

The following symmetry result is an important feature of the algorithm, and one that will simplify our own work.

\begin{proposition}[\cite{Scu63}]\label{prop:P w inverse is Q w}
For any permutation $w$, 
$$\P(w^{-1})=\Q(w).$$
\end{proposition}

Schensted's theorem~\cite[Theorem~1]{Sch61}, presented here as Theorem~\ref{thm:Greene's theorem}, articulates an important relationship between the RSK partition shape and the one-line notation for $w$.

\begin{theorem}\label{thm:Greene's theorem}
Given a permutation $w$, the length of the longest increasing (resp., decreasing) subsequence in the one-line notation of $w$ is the size of the first row (resp., column) of $\P(w)$.
\end{theorem}

Due to Schensted's theorem, we can see that a permutation is fully commutative if and only if its RSK partition has at most two rows. We denote the set of values in the second row of the RSK insertion tableau of a permutation $w$ by $\Rowtwo{w}.$
More generally, we denote the set of values in the second row (resp., first row) of any tableau $T$ by 
$$\RowTwo(T) \ \ \text{ (resp., }\Rowone(T)).$$ 

Next we list some basic features of RSK insertion, which we may use without specific mention in the future. The following lemma is a consequence of the definition of RSK insertion.
\begin{lemma}\label{lem: rsk_basics} 
Let $w\in S_n$, and suppose $b$ bumps $z$ in the RSK insertion process for $w$.
Then $b<z$ and $b$ appears to the right of $z$ in the one-line notation of $w$. 
\end{lemma}

For permutation $v\in S_n$ and value $q\in\{1, \ldots, n\}$, let $c_v(q)$ be the column of $\P(v)$ into which $q$ is first inserted. 
Let $\textup{LIS}_v(q)$ be the length of a longest increasing subsequence of $v$ that ends with $q$. 
The following is a key result we will reference in our analysis. 

\begin{lemma}[{\cite[Lemma 3.3.3]{Sagan}}]
\label{lem: c equals LIS}
For $v\in S_n$ and $q\in\{1,\ldots, n\}$, we have $c_v(q) = \textup{LIS}_v(q)$.
\end{lemma}

One consequence of Lemma~\ref{lem: c equals LIS} is that certain values must be part of every longest increasing subsequence of a permutation.

\begin{corollary}\label{cor:q in every LIS}
For a permutation $v$, if $q$ is the only value in $v$ inserted into column $c_v(q)$ of $\P(v)$, then $q$ is in every longest increasing subsequence in $v$.
\end{corollary}

The last result in this subsection highlights basic properties of RSK tableaux for fully commutative permutations.
\begin{lemma}
\label{lem:row 2 of P(w) appears in increasing order in w}
\label{lem:b1<b2<... and b1,b2,... apppear from left to right}
Let $w$ be a fully commutative permutation with $\Rowtwo{w} = \{z_1 < z_2 < \cdots < z_t\}$. For each $i\in \{1,\dots, t\}$, 
let $b_i$ be the value that bumps $z_i$ from the first row to the second row during the construction of $\P(w)$.  Then we have the following.
\begin{enumerate}[\hspace{.5cm}(a)]
\item The sequence $z_1 z_2 \cdots z_t$ is an (increasing) subsequence of $w$. In other words, the values $z_1, z_2,  \ldots, z_t$ appear from left to right in the one-line notation of $w$.  
\label{row_two_left_to_right}

\item 
The sets $\{z_1,\dots,z_t\}$ and $\{b_1,\dots,b_t\}$ are disjoint. 
In other words, during RSK insertion, no value can both bump something and be bumped by something. 
\label{bumper_cannot_be_bumped} 

\item The sequence $b_1 b_2 \dots b_t$ is an increasing subsequence of $w$.
\label{bump_left_to_right} 

\item Let $1\le i<j\le t$. During RSK insertion, the value $z_i$ is bumped before $z_j$.\label{bumped_left_to_right}
\end{enumerate}
\end{lemma}

\begin{proof}\mbox{}
\begin{enumerate}[\hspace{.5cm}(a)]
\item Suppose, to the contrary, that $z_i$ appears to the right of $z_{i+1}$ for some $i$. 
Since $b_i$ bumps $z_i$ during the insertion algorithm, we know $b_i<z_i$, and the value $b_i$ occurs to the right of $z_i$ in the one-line notation of $w$. 
This means $z_{i+1}z_ib_i$ is a $321$-pattern in $w$, which is a contradiction. 

\item By (\ref{row_two_left_to_right}), we have that $z_1 \cdots z_t$ is an increasing subsequence of $w$. Hence, there are no $i$ and $j$ such that $z_i$ bumps $z_j$, and the sets $\{z_1,\ldots, z_t\}$ and $\{b_1,\ldots, b_t\}$ are therefore disjoint.

\item First, we show that $b_1 < \cdots < b_t$.
Suppose, to the contrary, that $b_i > b_{i+1}$ for some $i$. Since $z_i$ appears to the left of $b_i$ in the one-line notation for $w$ and $z_i > b_i$, 
the value $b_{i+1}$ must appear to the left of $b_i$ in order to avoid a 321-pattern in $w$. We also know $z_{i+1}$ appears to the left of $b_{i+1}$ in the one-line notation for $w$ and $z_{i+1} > b_{i+1}$. 
From (\ref{row_two_left_to_right}), we know 
$z_i z_{i+1}$ is a subsequence of 
the one-line notation for $w$. 
Combining all of these observations, we conclude that
\[z_i z_{i+1} b_{i+1} b_i\] is a subsequence of the one-line notation of $w$. 
So, since $z_i$ is bumped by $b_i$, immediately before $b_{i+1}$ is inserted, the value $z_i$ is still in the first row. 
This means that $b_{i+1}$ must bump a number no larger than $z_i$, which contradicts the assumption that $b_{i+1}$ bumps $z_{i+1}$. 
Therefore $b_1 < \cdots < b_t$. 

Now say for some $i$ that  $b_{i+1}$ occurs to the left of $b_i$ in the one-line notation for $w$. Since $b_i<b_{i+1}$, we would have the 321-pattern $z_{i+1}b_{i+1}b_i$ in $w$, which is a contradiction. Hence $b_1, \ldots, b_t$ occur from left to right in $w$. 
\item This follows from (\ref{bump_left_to_right}).
\end{enumerate}
\end{proof}

\subsection{Characterization of boolean RSK tableaux}\label{subsection.boolean.rsk}

While Schensted's Theorem (Theorem~\ref{thm:Greene's theorem}) guarantees the insertion tableau of a boolean permutation has at most two rows, not every $2$-row standard tableau is the insertion tableau of some boolean permutation. For example, the tableau $T_1$ below is the insertion tableau of the boolean permutation $w=315264 = \rw{21435}\in S_6$, but $T_2$ cannot be obtained as the insertion tableau of any boolean permutation.
\[
T_1 = \begin{ytableau}
1 & 2 & 4\\
3 & 5 & 6
\end{ytableau}\,, \quad
T_2 = \begin{ytableau}
1 & 2 & 3\\
4 & 5 & 6
\end{ytableau}\,
\]

We review the characterization of these tableaux from~\cite{GPRT1}. First we need to define when a set of integers is ``uncrowded.''

\begin{definition}\label{defn:balanced} 
Let $L$ be a set of integers. If, for all integers $x$ and $y$, with $x>0$, we have 
$$|[y,y+2x] \cap L| \le x+1,$$ 
then we will say that $L$ is \emph{uncrowded}.
Otherwise, we say that $L$ is \emph{crowded}.
\end{definition}

Let $T$ be a standard tableau with at most two rows. When $\RowTwo(T)$ is uncrowded, we also call the tableau $T$ an \emph{uncrowded tableau}, and $T$ is a \emph{crowded tableau} otherwise. In the example above, we can see that $T_1$ is an uncrowded tableau because its second row $\{3,5,6\}$ is an uncrowded set, while $T_2$ is a crowded tableau because
\[
|[4,4+2\cdot 1] \cap \{4,5,6\}| = 3 > 1+1\,.
\]

The following proposition, which is the combination of several results in~\cite{GPRT1}, provides a characterization of RSK tableaux coming from boolean permutations.

\begin{proposition}\label{cor:second row rules for balanced tableaux}
A standard tableau $T$ with at most two rows is the insertion (or recording) tableau of a boolean permutation if and only if $T$ is uncrowded.
\end{proposition}

We define an {\it uncrowded (respectively, crowded) permutation} to be a permutation with an uncrowded (respectively, crowded) insertion tableau. By Proposition \ref{cor:second row rules for balanced tableaux}, a permutation is uncrowded exactly when it shares an insertion tableau with some boolean permutation.

\section{Fully commutative elements and the weak order} \label{sec:fully comm and the weak order}

From Theorem~\ref{thm:Greene's theorem}, we know that the RSK partition for a permutation has at most two rows if and only if the permutation is $321$-avoiding; that is, if and only if it is fully commutative. Boolean permutations, which avoid patterns $321$ and $3412$, are a special class of fully commutative permutations, and Proposition~\ref{cor:second row rules for balanced tableaux} fully characterized their RSK tableaux. 
In this section, we build upon Proposition~\ref{cor:second row rules for balanced tableaux} to study the insertion tableaux of fully commutative, but not necessarily boolean, permutations. 

\subsection{Boolean core}
\label{sec:boolean core}
We set the stage using the following lemma, which is little more than a restatement of the definition of fully commutative.

\begin{lemma}\label{lem:repeated letters in fully commutative}
Let $w$ be a fully commutative permutation and $\rw{u} \in R(w)$. If $j$ is a repeated letter in $\rw{u}$, then each pair of copies of $j$  must be separated by both $j+1$ and $j-1$ in $\rw{u}$. 
Put another way, if $x\prec y$ are elements of the heap $H_w$ both with label $j$ (i.e.,  $u_x=u_y=j$), then $H_w$ contains elements $p$ and $\tilde{p}$ with labels $k+1$ and $k-1$ such that $x\prec p \prec y$ and $x\prec \tilde{p} \prec y$. 
\end{lemma}

A key feature of boolean permutations is that their reduced words contain no repeated letters. This property fails to hold for arbitrary fully commutative permutations, but, as we will show in the next result, every fully commutative permutation can be thought of as having a ``boolean core.'' More precisely, we can write any fully commutative permutation as the product of two permutations, one of which is boolean with the same support as the original permutation. As a result, every fully commutative permutation has a reduced word in which any repetition of letters occurs only after every letter in the support has appeared. 

\begin{theorem}
\label{thm:fully commutative decomposes via boolean}
Let $w$ be a fully commutative permutation. Then we can write $w = \core{w} w'$, where $\ell(w) = \ell(\core{w}) + \ell(w')$, the permutation $\core{w}$ is boolean, and $\supp(\core{w}) = \supp(w)$. Furthermore, this $\core{w}$ is uniquely determined by $w$.
\end{theorem}

\begin{proof}
Fix a fully commutative permutation $w$ and $\rw{u} = \rw{u_1 \cdots u_\ell} \in R(w)$. Because $w$ is fully commutative, it has a unique heap $H_w$. 
Elements with the same label are comparable in $H_w$. Thus, for each $i\in \supp{(w)}$, we can take the smallest element $x$ in $H_w$ such that $u_x=i$. 
Let $C$ denote the set of all such smallest elements, for $i\in \supp{(w)}$.

We claim that $C$ is an order ideal of $H_w$, and we will show that this is true using a proof by contradiction. 
Suppose $x,y \in H_w$ such that 
$y \in C$ and
$x$ is covered by $y$.  
Let $u_x = j$, and so by Lemma~\ref{lem:label cover} $u_y = j\pm 1$.  
Suppose, for the purpose of obtaining a contradiction, that $x \not\in C$. Thus there exists $\tilde{x} \prec x$ with $u_{\tilde{x}} = j$. Then, by Lemma~\ref{lem:repeated letters in fully commutative}, there exist $p,\tilde{p} \in H_w$ such that $\tilde{x} \prec p\prec x$, $\tilde{x}\prec \tilde{p} \prec x$, $u_p = j+1$, and $u_{\tilde{p}} = j-1$. But then we would have $y \not\in C$, which is a contradiction.

Because $C$ is an order ideal of $H_w$, we can choose a labeled linear extension of $H_w$ whose first $|C|$ letters
are precisely $\supp(w)$. This produces a reduced word for $w$ whose leftmost $|C|$ letters are precisely $\supp(w)$. 

Finally we show that this $\core{w}$ is also unique. Recall that any prefix of a reduced word for $w$ corresponds to an order ideal of $H_w$. The condition $\supp{(\core{w})} = \supp{(w)}$ requires that we pick an order ideal of $H_w$ having $|\supp(w)|$ elements of distinct labels. Elements with the same label are comparable in $H_w$, meaning that we are forced to select the smallest one for each label.
\end{proof}

We refer to the boolean permutation $\core{w}$  
in Theorem~\ref{thm:fully commutative decomposes via boolean} as the \emph{(right) boolean core} of a fully commutative permutation, where ``right'' refers to the fact that $\core{w}$ is 
the maximal boolean permutation that is less than $w$ in the right weak order.

\begin{example}
The heap of the permutation $w=345619278$ in Example~\ref{ex:heap345619278} is given in Figure~\ref{fig:heap345619278}.
The boolean core of $w$ is $\core{w} = 314569278$, 
and its heap is given in
Figure~\ref{fig:fence_reduced_word21873456}. Note that the reduced word $\rw{21873456} \in R(\core{w})$ appears as the left prefix of the reduced word $\rw{21873456234} \in R(w)$.
\end{example}

Theorem~\ref{thm:fully commutative decomposes via boolean} can also be proved without the language of heaps, by inducting on the length of a permutation. 

\subsection{Containment under the weak order}

Theorem~\ref{thm:fully commutative decomposes via boolean} identifies the boolean core of a fully commutative permutation, which gives some sense of how fully commutative permutations can be viewed as ``elongations'' of boolean permutations. We can similarly consider lengthening a fully commutative permutation. This leads to an important property about insertion tableaux.

\begin{theorem}\label{thm:growing fully commutative means row 2 subsets}
Let $v$ and $w$ be fully commutative permutations such that $w = vs_i$ with $\ell(w) = \ell(v) +1$. Then $\Rowone(\P(v))\supseteq \Rowone(\P(w))$; equivalently, $\Rowtwo{v} \subseteq \Rowtwo{w}$.
\end{theorem}

\begin{proof}
Let $v$ and $w$ be as in the statement of the result. So
$$w = v(1) \cdots  v(i-1) \, v(i+1) \, v(i) \, v(i+2) \cdots v(n),$$ 
with $v(i)<v(i+1)$. The permutation $w$ is fully commutative by assumption, so it is $321$-avoiding. Therefore, in fact, we have 
\begin{equation}\label{eqn:inequalities about growing fully commutative}
\begin{aligned}
v(j) &< v(i+1) \text{ for all $1 \le j \le i$, and } \\ v(j) &> v(i) \text{ for all $i+1 \le j \le n$}.
\end{aligned}
\end{equation}

Set $\P_{i-1} := \P_{i-1}(v) = P_{i-1}(w)$ to be the insertion tableau for the shared prefix $v(1) \cdots v(i-1)$ in the two permutations. 
To compute $\P_{i+1}(v)$, we insert $v(i)$ first and then $v(i+1)$; 
to compute $\P_{i+1}(w)$, we insert $v(i+1)$ first and then $v(i)$.

Consider first what happens when we insert $v(i)$ into $\P_{i-1}$. There are two cases to consider: 
either $v(i)$ bumps something out of the first row of $\P_{i-1}$, or $v(i)$ gets appended to the end of the first row of $\P_{i-1}$.

Suppose first that $v(i)$ bumps some $z$ out of the first row of $\P_{i-1}$. 
For $v(i)$ to do this, the value $z$ must have been the smallest number in that row larger than $v(i)$.
To create $\P_{i+1}(v)$ from $\P_i(v)$, the value $v(i+1)$ must be appended to the first row of $\P_i(v)$, because $v(i+1) > v(j)$  for all $1 \le j \le i$, by \eqref{eqn:inequalities about growing fully commutative}. 
To construct $\P_{i}(w)$, we again have that $w(i) = v(i+1)$ gets appended to the end of the first row of $\P_{i-1}$. When $w(i+1) = v(i)$ is inserted into $\P_i(w)$, it must bump the smallest value in 
$\Rowone(\P_{i-1}) \cup \{ v(i+1)\}$ that is larger than $v(i)$; this value must be $z$, as above, because $z<v(i+1)$.
Therefore, 
$\P_{i+1}(v)=\P_{i+1}(w)$, with $\{v(i), v(i+1)\}$ in the top row and $z$ in the second row. 

Because the rest of the entries in the one-line notations of $v$ and $w$ are identical, we can conclude from here that $\P(v) = \P(w)$.

Now suppose, for the remainder of the proof, that when $v(i)$ is inserted into $P_{i-1}$ it  is appended to the end of the first row of $\P_{i-1}$. 
In other words, $v(i)$ is larger than all values in $\Rowone(\P_{i-1})$. 
Then, when $v(i+1)$ is inserted into $\P_i(v)$, this new value is also appended to the end of the first row because $v(i+1) > v(i)$. In other words, $\P_{i+1}(v)$ is created by appending both $v(i)$ and $v(i+1)$ to the first row of $\P_{i-1}$.

To construct $\P_i(w)$, on the other hand, we first insert $v(i+1)$. This gets appended to the end of the first row of $\P_{i-1}$ because $v(i+1)$ is larger than all other values seen so far, by \eqref{eqn:inequalities about growing fully commutative}. In contrast, $v(i) < v(i+1)$, so $v(i)$ will bump something out of the first row of $\P_i(w)$ in order to form $\P_{i+1}(w)$. Everything in $\Rowone(\P_{i-1})$
is greater than $v(i)$, so $v(i)$ must bump $v(i+1)$ itself. Therefore, $\Rowone(\P_{i+1}(v)) = \Rowone(\P_{i+1}(w)) \cup \{v(i+1)\}$. And more to the point, $\Rowone{(P_{i+1}(v))}\supset \Rowone{(P_{i+1}(w))}$.

Combining \eqref{eqn:inequalities about growing fully commutative} with the fact that $v(i)$ is larger than every letter in $\Rowone(\P_{i-1})$, we have that $v(i+1), \ldots, v(n)$ must each be larger than every letter in $\Rowone(\P_{i-1}) \cup \{v(i)\}$. 
Therefore, all future insertions performed during the computation of both $\P(v)$ and $\P(w)$ will not bump any letter of 
$\Rowone(\P_{i-1}) \cup \{ v(i) \}$ out of the first row. That is, everything in the first row from $v(i)$ leftward will remain unchanged during the remaining steps of the insertion algorithm.

We will prove that $\Rowone(\P(v))$
contains all of $\Rowone(\P(w))$, using an inductive argument with $\P_k(v)$ and $\P_k(w)$, for $i+1 \le k \le n$.
We have shown the base case:
$\Rowone{(\P_{i+1}(v))}\supset \Rowone{(\P_{i+1}(w))}$. Assume, inductively, that for some $k\geqslant i+1$, we have $\Rowone{(\P_{k}(v))}\supseteq \Rowone{(\P_{k}(w))}$. There are two ways for $v(k+1)$ to be inserted into $\P_k(v)$: either it gets appended to the end of the top row of the tableau, or it bumps some value $z$.
\begin{itemize}
    \item If $v(k+1)$ gets appended to $\Rowone(\P_k(v))$, then everything in $\Rowone(\P_k(v))$ is less than $v(k+1)$. 
    Because $\Rowone{(\P_{k}(v))}\supseteq \Rowone{(\P_{k}(w))}$,
    all numbers in $\Rowone(P_k(w))$ must also be less than $v(k+1)$. 
    Therefore, $\P_{k+1}(w)$ is formed from $\P_k(w)$ by appending $v(k+1)$ to the end of the first row as well, and thus $\Rowone(\P_{k+1}(v)) \supseteq \Rowone(\P_{k+1}(w))$.
    
    \item If $v(k+1)$ bumps some $z\in \Rowone(\P_k(v))$, then $z$ is the smallest value in $\Rowone(\P_k(v))$ that is larger than $v(k+1)$. 
    We must now consider whether or not $z$ was in $\Rowone{(\P_k(w))}$. If not, then there is nothing to worry about and we are done. On the other hand, if $z\in \Rowone(\P_k(w))$, then, because $\Rowone(\P_{k}(w)) \subseteq \Rowone(\P_{k}(v))$, this 
    $z$ must also be the smallest number in $\Rowone(P_k(w))$ that is larger than $v(k+1)$. 
    Therefore, when we insert $v(k+1)$ into $\P_k(w)$, we will also bump $z$.
\end{itemize}

Thus the induction holds at all stages of the insertion algorithm, and hence 
$\Rowone(\P(v)) \supseteq \Rowone(\P(w))$.
The tableaux have height at most $2$, and so 
$\Rowtwo{v}\subseteq \Rowtwo{w}$, as well.
\end{proof}

We highlight several facts relevant to upcoming arguments in Section \ref{sec:adding to balanced}.
\begin{remark}\label{rem:facts}
For $v$ and $w$ fully commutative permutations with $w=vs_i$, $\ell(w)=\ell(v)+1$, and $\P(v)\neq \P(w)$, the following are established within the proof of Theorem~\ref{thm:growing fully commutative means row 2 subsets}:

\begin{enumerate}[\hspace{.5cm}(a)]
\item $v(k)<v(i+1)$ for $k<i$, and $v(k)>v(i)$ for $k>i$;\label{fact.0}
\item the value $v(i)$ does not bump anything in $\P(v)$,
and $v(i)\in \Rowone(\P(v))$;\label{fact.1}
\item $v(i)$ bumps $v(i+1)$ in $\P(w)$, and $v(i)\in \Rowone(\P(w))$;\label{fact.2}
\item $\Rowone(\P(v))\cap [1,v(i)] = \Rowone(\P_{i}(v)) = \Rowone(\P_{i+1}(w)) = \Rowone(\P(w))\cap [1,v(i)]$. 
\label{fact.3} 
\end{enumerate}
\end{remark}

Because the length of the first row of a permutation's shape is determined by the length of a longest increasing subsequence in the permutation, we can use Theorem~\ref{thm:growing fully commutative means row 2 subsets} to characterize when the insertion tableaux of $v$ and $vs_i$ are unequal.

\begin{corollary}\label{cor:insertion tableaux unequal along covering relation}
Let $v$ and $w$ be fully commutative permutations such that $w = vs_i$, with $\ell(w) = \ell(v) + 1$. Then $\P(v) \neq \P(w)$ if and only if every longest increasing subsequence in $v$ uses both $v(i)$ and $v(i+1)$. In particular, when $\P(v) \neq \P(w)$, we have $|\Rowtwo{w}| = |\Rowtwo{v}|+1$.
\end{corollary}

\begin{proof}
Note that $|\Rowtwo{w} \setminus \Rowtwo{v}| \le 1$, because the length of the longest increasing subsequence changes by at most one after swapping adjacent values in a position.
Since $\Rowtwo{v} \subseteq \Rowtwo{w}$
by Theorem~\ref{thm:growing fully commutative means row 2 subsets}, we have that $\Rowtwo{v} \subsetneq \Rowtwo{w}$ if and only if the size of the first row of $\P(v)$ is one more than the size of the first row of $\P(w)$.
By Schensted's theorem (Theorem~\ref{thm:Greene's theorem}), this holds if and only if
the length of a longest increasing subsequence of $v$ is one more than the length of a longest increasing subsequence of $w$.
Swapping $v(i)$ and $v(i+1)$ changes this length if and only if 
every longest increasing subsequence in $v$ uses both $v(i)$ and $v(i+1)$.
It follows that when $\Rowtwo{v} \subsetneq \Rowtwo{w}$, the set $\Rowtwo{w}$ contains exactly one more element than $\Rowtwo{v}$.
\end{proof}

Theorem~\ref{thm:growing fully commutative means row 2 subsets} has other implications for the weak order on fully commutative elements.

\begin{corollary} Let $v$ and $w$ be fully commutative permutations.
\begin{enumerate}[\hspace{.5cm}(a)]
\item
If $v$ is less than $w$ in the right weak order, then $\Rowtwo{v} \subseteq \Rowtwo{w}$.

\item
If $v$ is less than $w$ in the left weak order, then 
$\RowtwoQ{v} \subseteq \RowtwoQ{w}$.
\end{enumerate}
\end{corollary}

\begin{proof}
Statement (a) follows immediately from Theorem~\ref{thm:growing fully commutative means row 2 subsets}. Statement (b) follows from (a) and Proposition~\ref{prop:P w inverse is Q w}. 
\end{proof}

There is another important implication of Theorem~\ref{thm:growing fully commutative means row 2 subsets}, in conjunction with Theorem~\ref{thm:fully commutative decomposes via boolean}. This allows us to show the relationship between the insertion tableaux of a fully commutative element and that of its boolean core.

\begin{corollary}\label{cor:fully commutative and boolean core tableaux}
Let $w$ be a fully commutative permutation and $\core{w}$ its boolean core. Then
$$\Rowone(\P(\core{w})) \supseteq \Rowone(\P(w)) \ \text{ and } \ \Rowtwo{\core{w}} \subseteq \Rowtwo{w}.$$
\end{corollary}

The following example illustrates this result.

\begin{example}\label{ex:boolean core and fc tableaux}
Let 
$v=41623785 = \rw{32154673}$, $w=vs_5=41627385 = \rw{321546735}$, and let $\core{v}=41263785 = \rw{3215467}$ denote their common boolean core.
We can see that $\core{v}< v < w$ in the right weak order.
The RSK insertion algorithm produces
\[\P(\core{v})=\P(v)=
\young(12358,467) 
\qquad \text{ and } \qquad
\P(w)=\young(1235,4678)\,. 
\]
We have $\Rowone{(P(\core{v}))}\supseteq \Rowone{(P(v))}\supseteq \Rowone{(P(w))}$ and $\Rowtwo{\core{v}}\subseteq \Rowtwo{v} \subseteq \Rowtwo{w}$.
\end{example}

\section{Insertion tableaux dynamics}\label{sec:adding to balanced} 

Throughout this section, we will restrict our attention to certain important scenarios, and we will highlight our assumptions for the reader in centered boxed text. To begin, we will assume throughout this section that
$$\framebox{ $v$ and $w$ are fully commutative permutations with $w = vs_i$ and $\ell(w) = \ell(v) + 1.$ 
}$$

In Theorem~\ref{thm:growing fully commutative means row 2 subsets}, we learned that 
$$\Rowtwo{v} \subseteq \Rowtwo{w}.$$
Corollary~\ref{cor:insertion tableaux unequal along covering relation} gave conditions that determine exactly when $\P(v)\neq \P(w)$ in terms of the longest increasing subsequences of $v$. We next want to understand the entries of these  tableaux when they are unequal. In particular, if $\P(v)\neq \P(w)$, is it possible for $\P(w)$ to be uncrowded? Said another way, if the insertion tableau changes along a covering relation in the right weak order, can the covering permutation be uncrowded? 
If $i \not\in \supp(v)$, then this could certainly be the case. Consider, for example, when $v$ is the identity. On the other hand, if $i \in \supp(v)$, then, as we shall see, the answer to the question is \emph{no}.

Recall our assumptions in this section: $v$ and $w$ are fully commutative permutations (that is, they avoid $321$) with $w = vs_i$ and $\ell(w) = \ell(v) + 1$.
Let $M$ and $m$ be the values defined in Lemma~\ref{lem:Ten12}:
$$\framebox{ $M := \max\{v(j) : j \le i\}$ \textup{ \;\;\; and \;\;\; }
 $m := \min\{v(j) : j \ge i+1\}.$ }$$
 Our first lemma shows these values are part of a 3142-pattern in $v$ whenever $i\in\supp{(v)}$.

\begin{lemma}\label{lem:3142 pattern in v}
Suppose 
$i \in \supp(v)$. Then $v$ has a $3142$-pattern formed by $M\, v(i)\, v(i+1)\, m$.
\end{lemma}

\begin{proof}
Because $i \in \supp(v)$, it follows from Lemma~\ref{lem:Ten12} that $m < M$, and $M \ge v(i)$ and $m \le v(i+1)$ by definition. Because $w=vs_i$ and $\ell(w) > \ell(v)$, we must have $v(i) < v(i+1)$. 

Next we argue that $M>v(i)$. Suppose $M=v(i)$. Then $m<M=v(i)<v(i+1)$, 
so $v(i+1)\,v(i) \,m$ will form a $321$-pattern in $w$, which is a contradiction. 
Therefore $M> v(i)$. Similarly we can show that $m<v(i+1)$.

Since $w$ cannot have a $321$-pattern, we also must have $M < v(i+1)$ and $m > v(i)$.
Therefore the subsequence $M\,v(i)\,v(i+1)\,m$ is a $3142$-pattern in $v$.
\end{proof}

In fact, $321$-avoidance, the maximality of $M$, and the minimality of $m$ force even more structure upon $v$.

\begin{corollary}\label{cor:3142 pattern in v has increasing runs}
Suppose $i \in \supp(v)$. Then
$$v = \cdots M \ a_1 \ \cdots \ a_h \ v(i) \ v(i+1) \ e_1 \ \cdots \ e_j \ m \cdots,$$
where
\begin{equation}\label{eqn:inequalities in v}
a_1 < a_2 < \cdots < a_h < v(i) < m < M < v(i+1) < e_1 < e_2 < \cdots < e_j.
\end{equation}
\end{corollary}

Let us now further suppose, for the remainder of this section, that
$$\framebox{ $i\in\supp{(v)}$ and $\P(v) \neq \P(w).$ }$$
Furthermore, we will
$$\framebox{ \text{maintain the notation established in Corollary~\ref{cor:3142 pattern in v has increasing runs}.} }$$

Corollary~\ref{cor:insertion tableaux unequal along covering relation} tells us that every longest increasing subsequence in $v$ must use both $v(i)$ and $v(i+1)$.
In particular, this means that $h \ge 1$ and $j \ge 1$.

Since $\RowTwo(\P(v))\subsetneq \RowTwo(\P(w))$, it also follows from Corollary~\ref{cor:insertion tableaux unequal along covering relation} that there is a unique value
$$e\in \Rowone(\P(v)) \cap \RowTwo(\P(w)).$$
 We will show that $e$ occurs after $v(i+1)$ in $v$, that $e>M$, and, finally, that $\RowTwo(P(w))$ is crowded as it contains too many integers 
in the interval $\{M, \dots ,e\}$.

The next sequence of lemmas describe certain values in the rows of $\P(v)$ and $\P(w)$. Recall for a permutation $v\in S_n$ and value $q\in\{1, \ldots, n\}$, we define $c_v(q)$ to be the column of $\P(v)$ into which $q$ is first inserted.

\begin{lemma}\label{lem:M is bumped}
In the construction of $\P(v)$ and $\P(w)$, the value
$M$ is bumped to the second row by one of $a_1, \ldots, a_h$.
\end{lemma}

\begin{proof}
By Remark~\ref{rem:facts}(\ref{fact.1}), $v(i)$ does not bump anything in $\P(v)$, so we have that $c_v(M) < c_v(v(i))$. Since $v(i) < M$, this means some value that occurs between $M$ and $v(i)$ in $v$ must bump $M$ in $\P(v)$. Because the values prior to $v(i)$ are unchanged in $w$, $M$ will be bumped by that same value in $\P(w)$. 
\end{proof}

Just as we can track $M$ in the RSK insertion algorithm, we can determine the role of $m$ in the construction of $\P(v)$.

\begin{lemma}\label{lem:v(i+1) bumped by m and v(i)}
The value $v(i+1)$ is bumped by $m$ in $\P(v)$. 
\end{lemma}
\begin{proof}
Corollary~\ref{cor:3142 pattern in v has increasing runs} and Remark~\ref{rem:facts}(\ref{fact.1}) tell us that, just before $m$ is inserted in the process of constructing $\P(v)$, the first row contains $v(i)<v(i+1)<e_1<\dots<e_j$ with no element between $v(i)$ and $v(i+1)$.
By Lemma~\ref{lem:3142 pattern in v}, $v(i)<m<v(i+1)$, so $m$ will bump $v(i+1)$ in $\P(v)$. 
\end{proof}

Next, we apply Remark~\ref{rem:facts} and Lemma~\ref{lem:v(i+1) bumped by m and v(i)} to determine the position of $e$ in the one-line notation for $v$.
\begin{lemma}\label{lem:b after v(i+1)}
The value $e$ occurs after $v(i+1)$ in $v$.
\end{lemma}

\begin{proof}
By Remark~\ref{rem:facts}(\ref{fact.3}), we have 
\[\Rowone(\P(v))\cap [1,v(i)] = 
\Rowone(\P(w))\cap [1,v(i)].\]
Since $e\in \Rowone(\P(v))$ and $e\notin\Rowone(\P(w))$, we know $e>v(i)$. By Remark~\ref{rem:facts}(\ref{fact.2}), $v(i)$ does not bump anything in $\P(v)$ and $v(i)\in\Rowone(\P(v))$. Thus we have $c_v(e)>c_v(v(i))$, and  $e$ occurs after $v(i)$ in $v$. By Lemma~\ref{lem:v(i+1) bumped by m and v(i)}, $v(i+1)\in\RowTwo(\P(v))$, so $e\neq v(i+1)$. Hence $e$ occurs after $v(i+1)$ in $v$.
\end{proof}

Using Lemma~\ref{lem:b after v(i+1)}, we can show that $e$ does not bump anything during the construction of $\P(v)$.

\begin{lemma}\label{lem: b doesnt bump}
The value $e$ does not bump anything in $\P(v)$. 
\end{lemma}

\begin{proof}
Suppose, for the purpose of obtaining a contradiction, that $e$ bumps something in $\P(v)$. Then there exists a value $q$ such that $e<q$ and $q$ occurs before $e$ in $v$. By Lemma~\ref{lem:b after v(i+1)}, $e$ occurs after $v(i+1)$ in $v$, so $q$ occurs before $e$ in $w$ as well. However, $e$ is bumped in $w$, so there is a value $q'$ with $q'<e$ and $q'$ occurring after $e$ in $w$. This yields a 321-pattern in both $v$ and $w$, which is not possible. Hence $e$ does not bump anything in $\P(v)$.
\end{proof}

Define $e_0 := v(i+1)$. By Corollary~\ref{cor:3142 pattern in v has increasing runs}, we see that  $c_v(e_k) = c_v(v(i+1))+k$ for $0\le k\le j$. For $k > j$, we can then define (if any) $e_k$ to be the first value in the one-line notation for $v$ with $c_v(e_k) = c_v(v(i+1))+k$. Let $r$ be maximal so that $\{e_0, e_1, \ldots, e_r\} \subseteq \Rowtwo{v}$.
For all $0\leq k\leq r$, let $t_0:=m, t_1, \ldots, t_r$ be the values such that $t_k$ bumps $e_k$ in $\P(v)$. By Lemma~\ref{lem:row 2 of P(w) appears in increasing order in w}(\ref{bump_left_to_right}) we have $t_0<t_1<\cdots<t_r$, and these values appear from left to right in the one-line notation of $v$.

For a permutation $v\in S_n$ and a value $q\in\{1,\ldots, n\}$, recall that we define $\textup{LIS}_v(q)$ to be the length of a longest increasing subsequence of $v$ that ends with $q$. The next lemma shows that the columns into which the values $t_k$ are first inserted are the same in $\P(v)$ and $\P(w)$, for $0\leq k\leq r$.

\begin{lemma}\label{lem:t_k column unchanged}
For all $0\leq k \leq r$, $c_v(t_k) = c_w(t_k)$.
\end{lemma}

\begin{proof}
By construction, we have $c_v(t_{k+1}) = c_v(t_k)+1$ and $c_w(t_{k+1}) > c_w(t_k)$ for $0 \le k < r$. Furthermore, since $\textup{LIS}_w(t_{k})\leq \textup{LIS}_v(t_{k})$, we know by Lemma~\ref{lem: c equals LIS} that $c_w(t_{k})\leq c_v(t_{k})$ for $0\le k \le r$. We prove the statement by induction on $k$.

First we show $c_v(t_0) = c_w(t_0)$. By Corollary~\ref{cor:3142 pattern in v has increasing runs} and Remark~\ref{rem:facts}(\ref{fact.2}), the first row of $\P_{i+j+1}(w)$ contains $v(i)$ and $e_1$, with no element between them. Because $v(i)<m<e_1$, we know that $m=t_0$ bumps $e_1$ in $\P(w)$. Since $c_w(e_1)=c_v(v(i+1))$ and $t_0$ bumps $v(i+1)$ in $\P(v)$, we have $c_w(t_0) = c_w(e_1) = c_v(v(i+1)) = c_v(t_0)$.

Next assume for some $0\leq k <r$ that $c_v(t_k)=c_w(t_k)$. Then we have 
\[
c_w(t_k)+1 \le c_w(t_{k+1}) \le c_v(t_{k+1}) = c_v(t_k)+1.
\]
Therefore $c_w(t_{k+1}) = c_v(t_{k+1})$, proving the statement.
\end{proof}

Since $e\in \Rowone(\P(v))$ occurs after $v(i+1)$ in $v$ and does not bump anything in $\P(v)$, it follows that $e=e_k$ for some $k>r$. Therefore the value $e_{r+1}$ exists, and by the definition of $r$, we have $e_{r+1}\in \Rowone(\P(v))$ with $e_{r+1}\leq e$. In fact, as a corollary to Lemma~\ref{lem:t_k column unchanged}, we can show that $e_{r+1}=e$.

\begin{corollary}\label{cor:b'=b}
We have $e_{r+1}\in\RowTwo(P(w))$, and so $e_{r+1}=e$.
\end{corollary}

\begin{proof}
Since $e_{r+1}$ is the only value in $v$ inserted into column $c_v(e_{r+1})$ of $\P(v)$, we can apply Corollary~\ref{cor:q in every LIS} to conclude that $e_{r+1}$ is in every longest increasing subsequence in $v$. By Corollary~\ref{cor:insertion tableaux unequal along covering relation}, this implies that every longest increasing subsequence in $v$ ending with $e_{r+1}$ must use both $v(i)$ and $v(i+1)$. As a result, $\textup{LIS}_w(e_{r+1})=\textup{LIS}_v(e_{r+1})-1$. By Lemma~\ref{lem: c equals LIS}, $c_w(e_{r+1}) = c_v(e_{r+1})-1$.  We know $c_v(e_{r+1})-1= c_v(t_r)$ by definition, and by Lemma~\ref{lem:t_k column unchanged}, $c_v(t_r) = c_w(t_r)$. Hence $c_w(e_{r+1})=c_w(t_r)$. Since $t_r<e_{r+1}$, we conclude that $t_r$ bumps $e_{r+1}$ in $\P(w)$. Since $e_{r+1}\in\Rowone(\P(v))$, it follows that $e_{r+1}=e$.
\end{proof}

Next we show that $e_r$ and $e$ are consecutive values.

\begin{lemma}\label{lem:b = b_r + 1}
With notation as above, $e = e_r + 1$.
\end{lemma}

\begin{proof}
Since $e$ occurs after $e_r$ and $e_r,e\in \RowTwo(P(w))$, Lemma~\ref{lem:row 2 of P(w) appears in increasing order in w}(\ref{bump_left_to_right}) shows that $e_r<e$. Suppose, for the purpose of obtaining a contradiction, that $e\neq e_r+1$, and so $e>e_r+1$.
We analyze where $e_r+1$ could occur in the one-line notation of $v$. First we argue that $e_r+1$ cannot occur after $e$. Suppose it occurs after $e$. Before $e_r+1$ is inserted into $\P(v)$, $e$ is in the first row and the element to the left of $e$ is either $e_r$ or $t_r$. Since $t_r<e_r<e_r+1<e$, the value $e_r+1$ will bump $e$, which contradicts the fact that $e\in \Rowone(\P(v))$.

Next we argue that $e_r+1$ cannot occur prior to $e_r$. Suppose $e_r+1$ is to the left of $e_r$. Before $e_r$ is inserted into $\P(v)$, if $e_r+1$ is in the first row, then $e_r$ will bump $e_r+1$, which contradicts Lemma~\ref{lem:row 2 of P(w) appears in increasing order in w}(\ref{bumper_cannot_be_bumped}). This forces $e_r+1 \in \Rowtwo{v}$, which, then, contradicts Lemma~\ref{lem:row 2 of P(w) appears in increasing order in w}(\ref{bumped_left_to_right}).

Therefore $e_r+1$ must occur after $e_r$ and before $e$, which implies $c_v(e_r)<c_v(e_r+1)<c_v(e)$. However, this is impossible since $c_v(e_r)+1=c_v(e)$. Hence $e=e_r+1$.
\end{proof}

The maximality of $M$ means that $M+1$ appears to the right of $v(i)$ in the one-line notation of $v$. Consider the set
$$[M+1, e_r] \setminus \{v(i+1) = e_0, e_1, \ldots, e_r\},$$
which has $e_r - (M+1) + 1 - (r+1)$ elements. These elements occur after $v(i)$ and are in $\Rowone(\P(v))$, so they must bump (some of) the $r$ elements $\{e_1, \ldots, e_r\}$ and nothing else, by definition of $r$. 
Therefore we get
$$e_r - M - (r+1) \le r,$$
and hence
\begin{equation}\label{eqn:b_r and M}
e_r - M \le 2r+1.
\end{equation}

Now consider the interval
$$I := [M, e].$$
This is a set of size $e - M + 1$, and we can use Lemma~\ref{lem:b = b_r + 1} and Equation~\eqref{eqn:b_r and M} to get
$$\big\vert I \big\vert = e - M + 1 = e_r + 1 - M + 1 \le 2r+3.$$
Moreover, the $(r+2)$-element set
$$ \{M, v(i+1) = e_0, e_1, \ldots, e_r\}$$
is a subset of $\Rowtwo{v}$.

We are now able to state the main result.

\begin{theorem}\label{thm:adding length and box but not support to uncrowded makes crowded}
Suppose that $v$ and $w$ are fully commutative permutations with $w = vs_i$, $\ell(w) = \ell(v) + 1$, and $i \in \supp(v)$. Suppose, moreover, 
that $\P(v) \neq \P(w)$. Then $w$ is a crowded permutation.
\end{theorem}

\begin{proof}
As discussed above, there are $r+2$ elements of the interval $I$ in $\Rowtwo{v}$, and the interval $I$ contains at most $(2r+3)$ elements. By Theorem~\ref{thm:growing fully commutative means row 2 subsets}, Corollary~\ref{cor:b'=b}, and Lemma~\ref{lem:b = b_r + 1}, there are $r+3$ elements of the interval $I$ in $\Rowtwo{w}$, which means that $w$ must be crowded.
\end{proof}

A corollary of this result is an alternate characterization of uncrowded permutations

\begin{corollary}\label{cor.uncrowded}
Let $w$ be a fully commutative permutation with boolean core $\core{w}$. Then $w$ is uncrowded if and only if $\P(\core{w})= \P(w)$.
\end{corollary}

\section{Minimal crowded permutations under the weak order} \label{sec:minimal crowded permutations}

Consider the poset of fully commutative (that is, $321$-avoiding) permutations in $S_n$ under the right weak order. The RSK partitions of such permutations have at most two rows, and we saw in Theorem~\ref{thm:growing fully commutative means row 2 subsets} that the content of their second rows obeys a subset relation along covering relations in the weak order. 
We also saw, in 
Proposition~\ref{cor:second row rules for balanced tableaux}, 
that a 2-row tableau is an insertion tableau for a boolean permutation if and only if it is an uncrowded tableau.

Recall that a fully commutative permutation $w$ is called ``uncrowded'' if its insertion tableau is an uncrowded tableau.
Otherwise a permutation is ``crowded.'' Set
\begin{align*}
    \balanced_n &:= \{w \in S_n \mid w \text{ is fully commutative and uncrowded}\}, \text{ and}\\
    \unbalanced_n &:= \{w \in S_n \mid w \text{ is fully commutative and crowded}\}.
\end{align*}

The subset relation in Theorem~\ref{thm:growing fully commutative means row 2 subsets} allows us to conveniently partition the fully commutative elements into  two sets: uncrowded and crowded permutations.

\begin{lemma}\label{lem:balanced order ideal NEW}
Consider the fully commutative elements of $S_n$, partially ordered according to the right weak order. The uncrowded permutations form an order ideal of this poset, and the crowded permutations form a dual order ideal of this poset.
\end{lemma}

\begin{proof}
This follows from Theorem~\ref{thm:growing fully commutative means row 2 subsets} and Proposition~\ref{cor:second row rules for balanced tableaux}.
\end{proof}

Thus we can identify this partition of the fully commutative permutations in $S_n$ by characterizing the maximal elements of the set
$\balanced_n$
or, equivalently, the minimal elements of the set
$\unbalanced_n.$
The minimal elements of this latter set satisfy a pattern containment condition.
Before we state and prove that property, consider what it means for $w$ to be a minimal element of $\unbalanced_n$: the fully commutative permutation $w$ is crowded, while every fully commutative permutation $ws_i$ that it covers is uncrowded. 

For the remainder of this section, we will assume that
$$\framebox{ $w$ is fully commutative; i.e., $w$ is $321$-avoiding. }$$
We begin by recalling a standard definition: an integer $d\in \{1, \dots , n-1\}$ is a \emph{descent} of 
$w\in S_n$ if $w(d) > w(d+1)$. 

\begin{lemma}\label{lem:non-bumping descent doesn't change tableau NEW}
Suppose that $d$ is a descent of $w$, and that $w(d+1)$ does not bump $w(d)$ during RSK insertion. Then $\P(w) = \P(w s_d)$.
In other words, if $\P(w) \neq \P(w s_d)$, then either $d$ is not a descent of $w$, or $w(d+1)$ bumps $w(d)$ during RSK insertion.
\end{lemma}

\begin{proof}
Set $v := w s_d$, and $\P' := \P_{d-1}(w) = \P_{d-1}(v)$. Because $w$ is $321$-avoiding and $d$ is a descent of $w$, the value $w(d)$ must be larger than everything to its left in the one-line notation of $w$. Thus, in forming $\P_d(w)$, this $w(d)$ gets appended to the end of the first row of $\P'$, without bumping anything. In forming $\P_{d+1}(w)$, the value $w(d+1)$, which is less than $w(d)$ because $d$ is a descent, will bump something. Let $z$ be the value that it bumps; i.e., $z$ is the smallest value in $\Rowone(\P_d(w))$ that is larger than $w(d+1)$. We know by assumption that $z \neq w(d)$. In particular, $z < w(d)$ and $z$ appears to the left of $w(d)$ in $w$. This last fact means that $z \in \Rowone(\P')$.

In forming $\P_d(v)$, the value $v(d) = w(d+1)$ bumps $z$ from the first row of $\P'$ to the second row. In forming $\P_{d+1}(v)$, the value $v(d+1) = w(d)$ is the largest value we have seen so far, so it gets appended to the end of the first row of $\P_d(v)$, without bumping anything. Therefore $\P_{d+1}(w) = \P_{d+1}(v)$, and hence $\P(w) = \P(v)$.
\end{proof}

Somewhat akin to Lemma~\ref{lem:non-bumping descent doesn't change tableau NEW}, we can make the following additional statement, which we phrase in terms of Knuth relations. 

\begin{definition}\label{defn:knuth relations NEW}
Two permutations $w$ and $v$ \emph{differ by one Knuth relation} if $w$ is the result of replacing a consecutive $312$-pattern in $v$ by a consecutive $132$-pattern (or vice versa), or replacing a consecutive $231$-pattern in $v$ by a consecutive $213$-pattern (or vice versa).
\end{definition}

\begin{lemma}\label{lem:small neighbor doesn't change tableau NEW}
Let $d$ be a descent of $w$. If $w(d+2) < w(d)$, then $\P(w) = \P(w s_d)$. 
\end{lemma}

\begin{proof}
The permutation $w$ is $321$-avoiding, so $w(d)w(d+1)w(d+2)$ must be a $312$-pattern. Knuth's theorem~\cite{Knuth70} says that the insertion tableau is preserved under a Knuth relation, so $\P(w) = \P(w s_d)$.
\end{proof}

We now return to the characterization motivated by Lemma~\ref{lem:balanced order ideal NEW}: identification of the minimal elements of $\crowded_n$ in the poset of fully commutative permutations of $S_n$.

\subsection{Consequences of minimality in $\crowded_n$}\label{sec:consequences of minimality}

As it turns out, knowing that a permutation is minimal in the dual order ideal $\crowded_n$ 
imposes substantial structure on the permutation.
In this subsection, we will collect many of these consequences of minimality, with the ultimate goal of proving a characterization of minimality in Section~\ref{sec:characterization of minimality}.

Throughout this subsection we will consider permutations that are minimal elements of the dual order ideal $\crowded_n$, and we will identify features of the permutations that follow from that property. 

We begin with an immediate corollary of Lemma~\ref{lem:non-bumping descent doesn't change tableau NEW}.

\begin{corollary}\label{cor:bump must be adjacent NEW}
Let $w$ be a minimal crowded permutation. 
\begin{enumerate}[\hspace{.5cm}(a)]
\item Then $d$ is a descent of $w$ if and only if $w(d+1)$ bumps $w(d)$ to the second row during RSK insertion. 
\label{item:adjacent bump}
\item 
Furthermore, every $w(d) \in \RowTwo(\P(w))$ is bumped by $w(d+1)$.
\label{item:bumps from adjacent}
\end{enumerate}
\end{corollary}

\begin{proof}
It remains to prove Part~(\ref{item:bumps from adjacent}). 
Suppose $w(d)$ is an element of $\RowTwo(\P(w))$, bumped by $w(j)$ with $j > d+1$. 
Part~(\ref{item:adjacent bump}) tells us that $d$ is not a descent of $w$. 
Because $w(d) > w(j)$, there exists a descent $d' \in [d+1,j-1]$. Lemma~\ref{lem:non-bumping descent doesn't change tableau NEW} implies that $w(d'+1)$ bumps $w(d')$ during RSK insertion. 
The values $w(j)$ and $w(d'+1)$ cannot be equal, so $d'$ is in fact in $[d+1,j-2]$.
However, the fact that $w(j) w(d'+1)$ is not a subsequence of $w$  violates Lemma~\ref{lem:b1<b2<... and b1,b2,... apppear from left to right}(\ref{bump_left_to_right}), and so in fact we must have $j = d+1$.
\end{proof}

Lemma~\ref{lem:small neighbor doesn't change tableau NEW} and  Corollary~\ref{cor:bump must be adjacent NEW} impose rules on the values that are unaffected by bumping during RSK insertion.

\begin{corollary}\label{cor:fixed points in minimal crowded NEW}
Let $w$ be a minimal crowded permutation, with first descent $d$ and last descent $d'$. Then the permutation $w$ fixes all $i \in [1,d-1] \cup [d'+2,n]$.
\end{corollary}

\begin{proof}
Suppose, first, that some $i < d$ is not fixed by $w$. Let $i$ be minimal with this property, and let $j$ be such that $w(j) = i$. Minimality of $i$ means that $j>i$, and that $j-1$ is a descent of $w$. By Corollary~\ref{cor:bump must be adjacent NEW}, the value $i$ must bump $w(j-1)$ to the second row during RSK insertion. Moreover, this minimality means that $w(d+1) \ge i$. To avoid $w(d) w(d+1) i$ forming a $321$-pattern in $w$, we must have that $w(d+1) = i$. Minimality of $i$ and the fact that $d$ is the first descent mean that $w(d+1) = i < w(i)< w(i+1)< \cdots <w(d)$, and so $i$ will actually bump $w(i)$ during RSK insertion, contradicting the assumption that $i<d$ and Corollary~\ref{cor:bump must be adjacent NEW}.

Now suppose that some $i > d'+1$ is not fixed by $w$. Let $i$ be maximal with this property, and let $j$ be such that $w(j) = i$. Maximality of $i$ means that $j < i$, and that $j$ is a descent of $w$. And, by Corollary~\ref{cor:bump must be adjacent NEW}, this $i$ must be bumped by $w(j+1)$ during RSK insertion. To avoid $w(j)w(d')w(d'+1)$ forming a $321$-pattern in $w$, we must have that $j = d'$.  Moreover, maximality of $i > d'$ means that $w(j+2) < i$, and so Lemma~\ref{lem:small neighbor doesn't change tableau NEW} contradicts the minimality of $w$.
\end{proof}

At this point, we have established several properties about the one-line representation of minimal elements of $\crowded_n$. In fact, we can go even further, showing that values
in the interval 
$[d,d'+1]$, in the language of Corollary~\ref{cor:fixed points in minimal crowded NEW} must be, in a sense, interwoven.

For the remainder of this subsection, define:
\begin{center}
\framebox{
\begin{minipage}{4.75in}
\vspace{.1in}
 $\RowTwo(\P(w)) = \{z_1 < \cdots < z_t\}$, and $b_i$ is the value that bumps $z_i$ to $\RowTwo(\P(w))$ during the construction of $\P(w)$, for each $i = 1, \ldots, t$.
 \vspace{.05in}
\end{minipage}}
\end{center}
We will also want to be able to refer to ``minimally crowded sets,'' and so for positive integers $x$ and $y$, we will write 
$$\framebox{ $S_{x,y} := \begin{cases}
\{y,y+1,y+2\} & \text{ if $x = 1$, and}\\
\{y, y+1, y+3, y+5, \ldots, y+2x-1, y+2x\} & \text{ if $x > 1$.}
\end{cases}$ }$$

\begin{lemma}\label{lem:long_subsequence NEW}
Let $w$ be 
minimal in $\crowded_n$. 
 Then
\[
z_1b_1z_2b_2\cdots z_tb_t 
\]
is a consecutive subsequence of the one-line notation for $w$.
\end{lemma}

\begin{proof}
By Corollary~\ref{cor:bump must be adjacent NEW}, each $z_ib_i$ is a consecutive subsequence. 
From Lemmas~\ref{lem: rsk_basics} and~\ref{lem:row 2 of P(w) appears in increasing order in w}, we know that $z_1z_2z_3\cdots z_t$  
and 
 $b_1b_2b_3\cdots b_t$,
are subsequences of the one-line notation for $w$.
We next prove that $b_j z_{j+1}$ is a subsequence of $w$ for all $1\leq j < t$.

Suppose, for the sake of contradiction, there is some $j$ such that $b_j$ appears to the right of $z_{j+1}$, as in
$$w = \cdots z_j \cdots z_{j+1} \cdots b_j \cdots b_{j+1}\cdots .$$
But $z_{j+1} > z_j > b_j$ by Lemma~\ref{lem: rsk_basics}, hence there exists a descent $d$ such that $w(d)$ occurs at or after $z_{j+1}$, and before $b_j$. 
By Corollary~\ref{cor:bump must be adjacent NEW}, this is impossible. 
Therefore, for all $1\le j<t$, we must have $b_j$ appearing to the left of $z_{j+1}$ in the one-line notation for $w$.

We now prove that this subsequence is consecutive. Suppose that some value $q \neq z_{i+1}$ follows $b_i$. To avoid $321$-patterns, we must have that $q < z_{i+1}$. If $q < z_i$, then Lemma~\ref{lem:small neighbor doesn't change tableau NEW} would produce a contradiction with the fact that $w$ is minimal in $\crowded_n$. Thus it remains only to consider when $q > z_i$. 

Because $q$ is necessarily in $\Rowone(\P(w))$ and $b_{i+1}$ bumps $z_{i+1}$, we have $q < b_{i+1} < z_{i+1}$. We have assumed $q > z_i$, so, in fact, $z_{i+1} \ge z_i+3$. The set $\RowTwo(\P(w)) = \{z_1 < z_2 < \cdots < z_t\}$ is crowded, so there exist positive integers $x$ and $y$ such that $S_{x,y}\subseteq \RowTwo(\P(w))$. Since $z_{i+1} \ge z_i + 3$, the values $z_i$ and $z_{i+1}$ cannot both be in $S_{x,y}$. Define $j$ and $j'$ so that $w(j) = z_i$ and $w(j') = z_{i+1}$, and there are two options.

\begin{itemize}
\item  
If $y+2x < z_{i+1}$, set $\tilde{w} := w s_{j'}$: 
\[
\tilde{w} = \cdots z_ib_i q \cdots b_{i+1}z_{i+1}\cdots \,.
\]
Thus $S_{x,y} \subseteq \{z_1<\cdots <z_i\} \subseteq \Rowtwo{\tilde{w}}$. Therefore $\tilde{w} < w$, and $\tilde{w}$  is crowded, contradicting the minimality of $w$.

\item
If $z_i < y$, then construct $\tilde{w} := w s_j$:
\[
\tilde{w} = \cdots b_iz_i q \cdots z_{i+1}b_{i+1}\cdots \,.
\]
Then $S_{x,y} \subseteq \{z_{i+1}< \dots < z_t\} \subseteq \Rowtwo{\tilde{w}}$. Therefore $\tilde{w} < w$ and $\tilde{w}$  is crowded, again contradicting the minimality of $w$.
\end{itemize}
Thus there can be no such $q$, and the subsequence $z_1b_1z_2b_2\cdots z_tb_t$ is consecutive in $w$.  
\end{proof}

Recall from Section~\ref{subsection.boolean.rsk} that a fully commutative permutation $w$ is crowded if and only if $\Rowtwo{w}$ is a crowded set.  This means that there exist positive integers $x$ and $y$ such that 
$$|[y,y+2x] \cap \Rowtwo{w}| > x+1.$$
In fact, 
we can choose $x$ and $y$ so that
$$S_{x,y}:= [y,y+2x] \cap \Rowtwo{w}.$$
Note that the set $S_{x,y}$ contains at least three elements.

\begin{lemma}\label{lem:c_3 < y_1 NEW}
Let $w$ be a crowded 
permutation, and let $S_{x,y} \subseteq \Rowtwo{w}$ be as described above. The value that bumps the third smallest element of $S_{x,y}$ during RSK insertion is less than $y$. 
\end{lemma}

\begin{proof}
Let $c$ be the value that bumps the third smallest element of $S_{x,y}$. By Lemma~\ref{lem:row 2 of P(w) appears in increasing order in w}(\ref{bumper_cannot_be_bumped}), we know $c \notin \RowTwo(\P(w))$. We will prove our result in two cases: $x=1$ and $x>1$. First consider $x=1$. Since $S_{1,y} \subseteq \RowTwo(\P(w))$ and $c<y+2$, we see that $c<y$.  

Now consider the case $x>1$, so 
$c=y+2$. By Lemma~\ref{lem:row 2 of P(w) appears in increasing order in w}, this implies $y+2k$ will bump $y+2k+1$ for all $1\leq k\leq x-1$. However, if $y+2x-2$ bumps $y+2x-1$, there is no value in $\Rowone{(\P(w))}$ both smaller than $y+2x$ and larger than $y+2x-2$ that could have bumped $y+2x$, and yet $S_{x,y} \subseteq \RowTwo(\P(w))$. Thus $c< y$.
\end{proof}

\begin{corollary}
\label{cor:unbalanced minima have 415263 NEW}
Let $w$ be a minimal element of $\unbalanced_n$. 
Then $w$ contains a consecutive occurrence of the pattern $415263$. 
Moreover, 
$w$ has an occurrence $w(i)\cdots w(i+5)$ of the pattern $415263$ in which $$\{w(i),\ldots,w(i+5)\} \cap \Rowtwo{w} = \{w(i),w(i+2),w(i+4)\}.$$
\end{corollary}

\begin{proof}
The permutation $w$ is crowded, so there exist integers $x$ and $y$ such that $$\{y,y+1,y+3,\ldots, y+2x-1,y+2x\} = [y,y+2x] \cap \Rowtwo{w}.$$ 
Since $ \Rowtwo{w}=\{z_1<\cdots< z_t\}$, there is some $1\leq r\leq t-2$ such that $y=z_r$. 

Lemma~\ref{lem:long_subsequence NEW} tells us that 
\[z_r b_r z_{r+1}b_{r+1}z_{r+2}b_{r+2}\] is a consecutive subsequence of $w$. 
By Lemma~\ref{lem:c_3 < y_1 NEW}, we know that $b_{r+2}<z_r$.
By Lemma~\ref{lem:row 2 of P(w) appears in increasing order in w}, we have $z_r<z_{r+1}<z_{r+2}$ and $b_r<b_{r+1}<b_{r+2}$.
Combining these with the fact that $b_i<z_i$ for each $i$,  the consecutive subsequence that $z_r b_r z_{r+1}b_{r+1}z_{r+2}b_{r+2}$ is a $415263$-pattern. By Lemma~\ref{lem:row 2 of P(w) appears in increasing order in w}, 
the values $b_r$, $b_{r+1}$, and $b_{r+2}$ are not in $\Rowtwo{w}$.
\end{proof}

In fact, we can say more about this set $S_{x,y}$. 

\begin{lemma}\label{lem:minimal crowded must use largest element NEW}
Let $w$ be minimal in $\crowded_n$.
Any crowded subset of $\RowTwo(\P(w))$ must include the largest element of $\RowTwo(\P(w))$. 
\end{lemma}

\begin{proof}
Recall that $z_t = \max\{\RowTwo(\P(w))\}$, and let $j$ be such that $w(j) = z_t$. The set $\Rowtwo{w}$ is crowded, so there are positive integers $x$ and $y$ such that $S_{x,y}\subseteq \Rowtwo{w}$. If $z_t\notin S_{x,y}$, then $S_{x,y} \subseteq \RowTwo(\P(w)) \setminus \{z_t\}$ and the permutation $\tilde{w} := w s_j < w$ would be crowded because $\Rowtwo{\tilde{w}} =\RowTwo(\P(w)) \setminus \{z_t\}$, contradicting the minimality of $w$. Thus 
every crowded subset $S_{x,y} \subseteq \Rowtwo{w}$ must have $y+2x = z_t$.
\end{proof}

This property about crowded subsets implies that when $w$ is minimal in $\crowded_n$, there is, in fact, a unique crowded subset of the second row of $\P(w)$ that is inclusion-wise minimal.

\begin{corollary}\label{cor:minimal means unique crowded NEW}
Let $w$ be minimal in $\crowded_n$. Then $\RowTwo(\P(w))$ contains exactly one inclusion-wise minimal crowded subset.
\end{corollary}

\begin{proof}
By Lemma~\ref{lem:minimal crowded must use largest element NEW}, every crowded subset of $\RowTwo(\P(w))$ includes the maximal element $z_t \in \RowTwo(\P(w))$. Let $S_{x,y}$ be the crowded subset for which $y$ is maximal.  Then $\{y,y+1\} \in \RowTwo(\P(w))$.  To avoid $\{y-1,y,y+1\}$  contradicting Lemma~\ref{lem:minimal crowded must use largest element NEW}, we must have $y-1 \not\in \RowTwo(\P(w))$. This means that there are no crowded sets $S_{x',y'}$ for $y' < y$.
\end{proof}

In Lemma~\ref{lem:long_subsequence NEW}, we proved the consecutivity of the subsequence $z_1b_1\cdots z_tb_t$ in $w$. We already know several inequalities among these letters, and there is now one more that we can establish.

\begin{lemma}\label{lem:minimal crowded has z_i < b_i+3 NEW}
Let $w$ be minimal in $\crowded_n$. 
For all $i \in [1,t-3]$, we have $z_i < b_{i+3}$.
\end{lemma}

\begin{proof}
By Corollary~\ref{cor:fixed points in minimal crowded NEW}, we can assume, without loss of generality, that $b_1 = 1$ and $z_t = n$. Furthermore, Corollary~\ref{cor:minimal means unique crowded NEW} forces $z_{t-1} = n-1$. Now suppose, for the purpose of obtaining a contradiction, that there exists $i \in [1,t-3]$ such that $z_i > b_{i+3}$.

The set $\{z_i, \ldots, z_{t-1}\}$ contains $t-i$ elements. Because $z_i < z_{i+1} < \cdots$, we have
$$\{z_i, \ldots, z_{t-1}\} \subseteq [z_i, n-1].$$
The interval $[z_i, n-1]$ can be partitioned into 
$$\{z_i, \ldots, z_{t-1}\} \sqcup \{b_j : b_j > z_i\},$$
meaning that the cardinality of $[z_i,n-1]$ is at most $(t-i) + (t-(i+3)) = 2(t-i) - 3$. Therefore $\{z_i, \ldots, z_{t-1}\}$ is crowded, contradicting Corollary~\ref{lem:minimal crowded must use largest element NEW}.
\end{proof}

From this property, we learn how $415263$-patterns can appear in a minimal element of $\crowded_n$.

\begin{corollary}\label{cor:minimal crowded has only consecutive 415263} 
Let $w$ be minimal in $\crowded_n$. Every $415263$-pattern in $w$ is consecutive.
\end{corollary}

\begin{proof}
Lemma~\ref{lem:row 2 of P(w) appears in increasing order in w}, Corollary~\ref{cor:fixed points in minimal crowded NEW}, and Lemma~\ref{lem:minimal crowded has z_i < b_i+3 NEW} mean that $w(i) < w(j)$ for all $j \ge i+6$, and so it is impossible to find a non-consecutive $415263$-pattern in $w$.
\end{proof}

In fact, any consecutive subsequence of a minimal element of $\crowded_n$ that both begins and ends with a descent must have one of two forms.

\begin{lemma}\label{lem:minimal means 315264 and 415263 patterns}
Let $w$ be a minimal element in $\crowded_n$, with descent set $\{d, d+2, \ldots, d+2k\}$. For every $i \in [0,k-2]$, the consecutive subsequence
$$w(d+2i) \cdots w(d+2i+5)$$
is either a $415263$- or a $315264$-pattern.
\end{lemma}

\begin{proof}
From Lemmas~\ref{lem:b1<b2<... and b1,b2,... apppear from left to right} and~\ref{lem:long_subsequence NEW}, it remains to show that in any such sequence
$$z_i b_i z_{i+1} b_{i+1} z_{i+2} b_{i+2},$$
we have that $z_i$ is greater than $b_{i+1}$. Suppose, for the sake of contradiction, that $z_i < b_{i+1}$, with $j$ defined so that $w(j) = z_i$. In other words, $\{w(1), \ldots, w(j+1)\} = [1,j+1]$. The permutation $w$ is crowded, so let $S_{x,y} \subseteq \RowTwo(\P(w))$ be the unique containment-wise minimal crowded set guaranteed by Corollary~\ref{cor:minimal means unique crowded NEW}. If $z_{i+1} \not\in S_{x,y}$, then $z_i \not\in S_{x,y}$, by Lemma~\ref{lem:minimal crowded must use largest element NEW}. Moreover, $\RowTwo(w s_j) = \RowTwo(w) \setminus \{z_i\}$, so $S_{x,y} \subseteq \RowTwo(w s_j)$, and $w s_j < w$ is an element of $\crowded_n$, contradicting the assumption of minimality.

Now suppose, on the other hand, that $z_{i+1} \in S_{x,y}$. Then $z_{i+1} + 1 \not\in S_{x,y}$, so it must be that $z_{i+1} + 2 \in S_{x,y}$, and this can only happen if $z_{i+1} + 1$ bumps $z_{i+1} + 2$ to $\RowTwo(w)$. In fact, a similar argument shows that $z_{i+1} + 2i$ might be bumped by $z_{i+1} + 2i-1$, but $z_{i+1} + 2i+1$ cannot then also be bumped, contradicting the fact that $S_{x,y}$ is crowded.
\end{proof}

\subsection{Characterization of minimality in $\crowded_n$.}\label{sec:characterization of minimality}

Having established a variety of properties of minimal elements of $\crowded_n$ in Section~\ref{sec:consequences of minimality}, we are now able to completely characterize those elements.

\begin{theorem}\label{thm:minimal crowded iff NEW}
A permutation $w$ is a minimal element of $\crowded_n$ if and only if it satisfies the conditions below. 
\begin{enumerate}[\hspace{.5cm}(a)]
    \item \label{item:descent set} The set of descents of $w$ has the form $\{d,d+2, d+4, \ldots, d+2k\}$ for some $k \ge 2$.
    \item \label{item:crowded set} The set $\{w(d), w(d+2), \ldots, w(d+2k)\}$ is crowded.
    \item \label{item:fixed points} The permutation fixes all $i \in [1,n] \setminus [d,d+2k+1]$.
    \item \label{item:415263} The pattern $415263$ occurs in $w$, and every occurrence of $415263$ is consecutive.
    \item \label{item:415263 and 315264} For each $i \in [0,k-2]$, the consecutive subsequence
    $$w(d+2i) \cdots w(d+2i+5)$$
    is either a $415263$- or a $315264$-pattern.
\end{enumerate}
\end{theorem}

\begin{proof}
First suppose that $w$ is a minimal element of $\crowded_n$. Then Corollary~\ref{cor:bump must be adjacent NEW} and Lemma~\ref{lem:long_subsequence NEW} establish Properties~(\ref{item:descent set}) and~(\ref{item:crowded set}). Property~(\ref{item:fixed points}) is proved in Corollary~\ref{cor:fixed points in minimal crowded NEW}, and Property~(\ref{item:415263}) is a result of 
Corollaries~\ref{cor:unbalanced minima have 415263 NEW} and~\ref{cor:minimal crowded has only consecutive 415263}. Finally, Property~(\ref{item:415263 and 315264}) follows from Lemma~\ref{lem:minimal means 315264 and 415263 patterns}. Finally, we know from Property~(\ref{item:descent set}) and Corollary~\ref{cor:bump must be adjacent NEW} that $\RowTwo(\P(w)) = \{w(d),w(d+2),\ldots,w(d+2k)\}$.

Now suppose that a permutation $w$ has Properties~(\ref{item:descent set})--(\ref{item:415263 and 315264}) in the statement of the theorem. It follows from (\ref{item:descent set}) and~(\ref{item:fixed points}) and the fact that $w$ is $321$-avoiding that $\RowTwo(\P(w)) = \{w(d),w(d+2),\ldots,w(d+2k)\}$. 

This  
and Property~(\ref{item:crowded set}) mean that $w \in\crowded_n$. It remains, now, to prove that $w$ is minimal in that set.

Suppose, for the purpose of obtaining a contradiction, that $w$ is not minimal in $\crowded_n$. In fact, suppose that $w$ is minimal with this property, meaning that anything covered by $w$ is either not crowded, or minimal in $\crowded_n$. In particular, there must be at least one $v = ws_i$ in the latter category, by our assumption about $w$. Given Property~(\ref{item:fixed points}), let us assume, without loss of generality, that $d=1$ and $d+2k = n-1$. Because $v < w$, we have that $w(i) > w(i+1)$. Moreover, our assumptions about $w$ mean that $w(2j-1) > w(2j)$ for all $j$, and each $w(2j)$ bumps $w(2j-1)$ to $\RowTwo(\P(w))$. In particular, $w(n-1) = n$ and $w(n-3) = n-1$. On the other hand, Properties~(\ref{item:descent set}), (\ref{item:415263}), and~(\ref{item:415263 and 315264}) mean that in $v$, those bumping rules are no longer the case when $2j-1 > i$. Indeed, in $v$, it is $w(2j+2)$ that bumps $w(2j-1)$ to $\RowTwo(\P(v))$ when $2j-1 \ge i$. Thus
$$\RowTwo(\P(v)) = \RowTwo(\P(w)) \setminus \{w(n-1)\}.$$
Since we have assumed that $v \in \crowded_n$ is minimal, we know from previous results that $\RowTwo(\P(v))$ contains $n-1$, $n-2$ (which would have been $w(n-5)$, and either $n-3$ or $n-4$ (which would have been $w(n-7)$.
\begin{itemize}
    \item If $n-3 \in \RowTwo(\P(v))$, then $w(n) < n-3$, and so $w(n-7)w(n-6)w(n-5)w(n-4)w(n-1)w(n)$ would be a non-consecutive $415263$-pattern, violating Property (\ref{item:415263}).
    \item If, instead, $n-4 \in \RowTwo(\P(v))$, then $n-4 = w(n-7)$ and hence $n-3 = w(n)$. If we try to understand the rest of $w$ while satisfying Property (\ref{item:415263}), we find that $n-5 = w(n-2)$, $n-6 = w(n-9)$, $n-7 = w(n-4)$, $n-8 = w(n-11)$, and so on, meaning that the set $\RowTwo(\P(v))$ will never actually be crowded.
\end{itemize}
Thus there can be no such $v<w$, and so $w \in \crowded_n$ is minimal.
\end{proof}

\begin{remark}
Continuing the notation of Theorem~\ref{thm:minimal crowded iff NEW}, it followed immediately that if $w$ is a minimal element of $\crowded_n$ then $\RowTwo(\P(w)) = \{w(d),w(d+2),\ldots,w(d+2k)\}$.
\end{remark}

\begin{example}
\label{ex:w=41627385 is minimal crowded}
The permutation $w=41627385$ is a minimal element of $\crowded_8$. We check each of the conditions to confirm this.
\begin{enumerate}[\hspace{.5cm}(a)]
\item The descents of $w$ are $\{1,3,5,7\}$, so $d=1$ and $k=3$. 
\item The set $\{w(1), w(3), w(5), w(7)\}$ is $\{4,6,7,8\}$, which is crowded due to $\{6,7,8\}$.
\item The third condition holds vacuously.
\item 
The permutation $w$ contains two occurrences of the  $415263$ pattern: $416273$ and $627385$. Both are consecutive subsequences of $w$.
\item We check $i = 0,1$: the subsequence $416273$ is a $415263$-pattern and the subsequence $627385$ is a $415263$-pattern. 
\end{enumerate}
The insertion tableau in this case is 
$$\P(w) = \young(1235,4678)\,,$$
and indeed $\Rowtwo{w} = \{w(1),w(3),w(5),w(7)\}$.
\end{example}

\section*{Acknowledgements}
The authors would like to thank the 2021--2022 Research Community in Algebraic Combinatorics program at ICERM, through which this research took place. We thank the organizers and staff 
for putting together this invigorating and inspiring workshop series. 
The authors are also grateful to Carolina Benedetti, for helpful discussions. 
Finally, this work benefited from computation using {\sc SageMath}~\cite{sage}.


\begin{thebibliography}{GPRT22}

\bibitem[BB05]{BB05}
A.~Bj\"{o}rner and F.~Brenti.
\newblock {\em Combinatorics of {C}oxeter groups}, volume 231 of {\em Graduate
  Texts in Mathematics}.
\newblock Springer, New York, 2005.

\bibitem[BJS93]{billey-jockusch-stanley}
S.~C. Billey, W.~Jockusch, and R.~P. Stanley.
\newblock Some combinatorial properties of {S}chubert polynomials.
\newblock {\em J. Algebraic Combin.}, 2:345--374, 1993.

\bibitem[Dev21]{sage}
The~Sage Developers.
\newblock {\em {S}age {M}athematics {S}oftware ({V}ersion 9.3)}.
\newblock The Sage Development Team, 2021.

\bibitem[GPRT22]{GPRT1}
E.~Gunawan, J.~Pan, H.~M. Russell, and B.~E. Tenner.
\newblock Runs and {RSK} tableaux of boolean permutations, 2022.
\newblock Preprint \arxiv{2207.05119}.

\bibitem[Knu70]{Knuth70}
D.~E. Knuth.
\newblock Permutations, matrices, and generalized {Y}oung tableaux.
\newblock {\em Pacific J. Math.}, 34:709--727, 1970.

\bibitem[MPPS20]{mpps}
J.~Morse, J.~Pan, W.~Poh, and A.~Schilling.
\newblock A crystal on decreasing factorizations in the 0-{H}ecke monoid.
\newblock {\em Electron. J. Comb.}, 27:2, 2020.

\bibitem[Nad15]{nadeau}
P.~Nadeau.
\newblock On the length of fully commutative elements.
\newblock {\em Transactions of the American Mathematical Society},
  370:5705--5724, 2015.

\bibitem[Sag01]{Sagan}
B.~E. Sagan.
\newblock {\em The symmetric group}, volume 203 of {\em Graduate Texts in
  Mathematics}.
\newblock Springer-Verlag, New York, second edition, 2001.
\newblock Representations, combinatorial algorithms, and symmetric functions.

\bibitem[Sch61]{Sch61}
C.~Schensted.
\newblock Longest increasing and decreasing subsequences.
\newblock {\em Canadian J. Math.}, 13:179--191, 1961.

\bibitem[Sch63]{Scu63}
M.~P. Sch\"{u}tzenberger.
\newblock Quelques remarques sur une construction de {S}chensted.
\newblock {\em Math. Scand.}, 12:117--128, 1963.

\bibitem[Sta99]{sta99}
R.~P. Stanley.
\newblock {\em Enumerative Combinatorics, {V}olume 2}, volume~62 of {\em
  Cambridge Studies in Advanced Mathematics}.
\newblock Cambridge University Press, Cambridge, first edition, 1999.

\bibitem[Sta12]{Sta12}
R.~P. Stanley.
\newblock {\em Enumerative combinatorics. {V}olume 1}, volume~49 of {\em
  Cambridge Studies in Advanced Mathematics}.
\newblock Cambridge University Press, Cambridge, second edition, 2012.

\bibitem[Ste96a]{stembridge96}
J.~R. Stembridge.
\newblock On the fully commutative elements of coxeter groups.
\newblock {\em Journal of Algebraic Combinatorics}, 5:353--385, 1996.

\bibitem[Ste96b]{Ste96}
J.~R. Stembridge.
\newblock On the fully commutative elements of {C}oxeter groups.
\newblock {\em J. Algebraic Combin.}, 5(4):353--385, 1996.

\bibitem[Ste98]{stembridge98}
J.~R. Stembridge.
\newblock The enumeration of fully commutative elements of coxeter groups.
\newblock {\em Journal of Algebraic Combinatorics}, 7:291--320, 1998.

\bibitem[Ten07]{tenner-patt-bru}
B.~E. Tenner.
\newblock Pattern avoidance and the {B}ruhat order.
\newblock {\em J.~Comb.~Theory Ser.~A}, 114:888--905, 2007.

\bibitem[Ten12]{Ten12}
B.~E. Tenner.
\newblock Repetition in reduced decompositions.
\newblock {\em Adv. Appl. Math.}, 40:1--14, 2012.

\end{thebibliography}

\end{document}